\documentclass[twoside,reqno,11pt]{amsart}
\usepackage{layout}
\usepackage{amssymb}
\usepackage[mathscr]{eucal}
\usepackage{hyperref}

\numberwithin{equation}{section}

\theoremstyle{plain}
\newtheorem{prop}{Proposition}[section]
\newtheorem{theorem}[prop]{Theorem}
\newtheorem{cor}[prop]{Corollary}
\newtheorem{lemma}[prop]{Lemma}

\theoremstyle{definition}
\newtheorem{defi}[prop]{Definition}
\newtheorem{con}[prop]{Convention}

\theoremstyle{remark}
\newtheorem{example}[prop]{Example}
\newtheorem{bem}[prop]{Remark}

\usepackage{amsfonts}
\usepackage{amsmath}
\usepackage[arrow, matrix, curve]{xy}
\usepackage{enumerate}

\newcommand{\R}{\mathbb{R}}
\newcommand{\C}{\mathbb{C}}
\newcommand{\N}{\mathbb{N}}
\newcommand{\psd}{\operatorname{P}_{3,4}}
\newcommand{\form}{\operatorname{H}}

\newcommand{\spann}{\operatorname{span}}
\newcommand{\longto}{\longrightarrow}

\newcommand{\GL}[1][3]{\operatorname{PGL}_{#1}(\R)} 
\newcommand{\G}{\mathcal{G}}
\newcommand{\zero}{\mathcal{Z}}
\newcommand{\proj}{\mathbb{P}}
\newcommand{\ord}{\operatorname{ord}}

\newcommand{\Disk}{\operatorname{D}}
\newcommand{\stab}{\operatorname{stab}}
\newcommand{\diag}{\operatorname{diag}}
\newcommand{\typeA}{\mathcal{A}}
\newcommand{\typeB}{\mathcal{B}}
\newcommand{\typeC}{\mathcal{C}}
\newcommand{\V}{\mathbb{V}}
\newcommand{\Var}{\mathcal{V}}
\newcommand{\A}{\mathbb{A}}
\newcommand{\bu}[1]{\widetilde{#1}}
\newcommand{\inp}{\operatorname{inp}}
\renewcommand{\S}{\textsf{\textit{S}}}
\newcommand{\Li}{\textsf{\textit{L}}}
\newcommand{\Qu}{\textsf{\textit{Q}}}
\newcommand{\T}{\textsf{\textit{T}}}
\newcommand{\D}{\textsf{\textit{D}}}
\newcommand{\qpid}{\mathfrak{q}}

\begin{document}
\author{Aaron Kunert}
\thispagestyle{empty}
\title{Facial structure of the cone of\\ nonnegative ternary quartics}
\begin{abstract}
In this work we will discuss the facial structure of the cone of nonnegative ternary quartics with real coefficients. 
We will establish an equivalence relation on the set of all faces, which preserves certain properties like dimension 
or the number of common zeros. 
Moreover we will give a complete list of equivalence classes and we will discuss their dimension and inclusions
between them. Eventually we will be able to present a complete lattice of all faces of this particular cone.  
\end{abstract}
\maketitle
\thispagestyle{empty}

\section{Introduction}
We write $\proj^k=\proj(\C^{k+1})$ for the $k$--dimensional projective space over $\C$. Let $\proj^k(\R)$ denote the real 
points of $\proj^k$.  Let $\form_{k}\subseteq \R[x,y,z]$ be the vector space of real polynomials in three variables 
homogeneous of degree $k$.
For $I \subseteq \form_k$ an arbitrary subset we write  
$\zero(I):=\{ x \in \proj^{2}(\R) \colon \ f(x)=0, \ \forall f \in I\}$ 
for the set of all \emph{real} zeros of $I$.
We set
\[
\psd = \left\{ p \in \form_{4} \colon \ p(x) \geq 0, \ \text{for all} \ \ x \in \proj^{n}(\R) \right\} 
\]
and call $P_{3,4}$ the \emph{cone of nonnegative ternary quartics}. Evidently $\psd$ is closed under addition and 
multiplication with nonnegative reals. Hilbert (\cite{hilbert})
proved in 1888 that every polynomial in $\psd$ is a sum of squares of quadratic forms. 
Yet still little is known about the facial structure of cones of nonnegative forms. 
In this work we will explore the facial structure 
of $\psd$ and show that $\psd$ has up to the natural $\GL$--action only finitely many faces. We will give a complete list of 
these, and discuss issues as exposednes and inclusions between faces.

\section{Basic concepts} 
\subsection{Convex Cones} 
\begin{defi}
\label{facedef}
A set $C \subseteq \R^k $ is called a \emph{convex cone}, if it is closed under addition and multiplication with nonnegative 
reals. A subcone $F \subseteq C$ is called \emph{face} of $C$, if for any $a,b 
\in C$, whenever $a+b \in F$, we must have $a,b \in F$. For a face $F \subseteq C$ we define the dimension of $F$ by 
\[ \dim F := \dim \spann(F). \] An element $u \neq 0$ contained in a one-dimensional face is called extremal element, for  
 $C=\psd$ it is also called extremal form.
Let $V$ be a vector subspace of $\R^k$. We call $F$ \emph{full} in $V$, if $V=\spann(F)$, or equivalently 
if $F=C\cap V$ and there exists $f \in F$ such that for all $g \in V$ there exists $\epsilon >0$ with $f+\epsilon g \in C$. 
To every $f \in C$ we attach the set 
\[ F_f := \left\{ g \in C \colon  \ f-\epsilon g \in C \ \text{for some} \  \epsilon>0 \right\}. \]
A face $F \subseteq C$ is called exposed if there exists a linear functional $L:\R^k \longto \R$, such that $L$ is 
nonnegative on $C$ and $F=\ker(L)\cap C$. In this case we say that $L$ defines $F$.  
\end{defi} 

\begin{bem}
It is easy to check that $F_f$ is a face of $C$. It is the unique face containing $f$ as an inner point. Since every face of 
$C$ has an inner point, every face $F \subseteq C$ can be written as $F_f$ for a suitable $f \in C$. 
By construction $g \in C$ is an inner point of $F_f$ if and only if $F_g=F_f$. 
If $C = \psd$ then $f \in F_g$ implies $\zero(g)\subseteq \zero(f)$. In particular $F_f=F_g$ only if 
$\zero(f)=\zero(g)$. 
\end{bem} 

\subsection{Infinitely near points} \mbox{} \\
Let $f \in \form_{4}$, $p \in \proj^2$ with $\ord_p(f)\geq 2$. We choose affine coordinates $x,y$ in $p$ (i.e. $p$ 
becomes the origin). We choose projective coordinates $z,w$ of $\proj^1$ and write $X:=\Var(xw-yz) \subset \A^2 \times 
\proj^1$ for the blowup of $\proj^2$ in $p$ with the corresponding projection $\pi \colon X \longto \A^2$.
We denote the ideal of the exceptional line by $\qpid:=(x,y) \subset \R[x,y,z,w]$. For $f$ as above we define the ideal 
\[ T:=\left((\pi^*(f),xw-yz):\qpid^2 \right), \]  
which is homogeneous in $z,w$ and defines a curve in $X$. 

\begin{bem}
If we have $\ord_p(f)=2$ in above construction then $\Var(T)$ is the strict transform of $f$. If $\ord_p(f)>2$ then 
$\Var(T)$ is the union of the strict transform and the exceptional line. 
\end{bem}  

\begin{defi}
A (first order) real infinitely near point of $f$ in $p$ is a real point on $\Var(T+q)\subseteq \proj^1$. We denote the set 
of all first order real infinitely near points of $f$ in $p$ by $\inp_f(p)$. 
\end{defi}

\begin{defi}
For a line $l$ through $p$ we write $l_p \in \proj^1$ for the intersection of the strict transform of $l$ in $p$ with 
the exceptional line $\Var(\qpid)$. Each point on the exceptional line can be written as $l_p$ for a suitable line $l$.  
\end{defi} 

\begin{defi}
For a fixed open embedding $\phi\colon A \lhook\joinrel\relbar\joinrel\rightarrow X$, i.e. affine coordinates on $X$, we 
denote by $\bu{f}$ the generator of $\phi^*(T)$.  
\end{defi}

\begin{bem}
\label{cbu}
In this paper we will always choose the the open embedding 
\begin{align*}
\A^2 &\longto X,\\
(x,y) &\longmapsto (x,xy;1:y). 
\end{align*}
Under this choice we have $\bu{f}(x,y)=\frac{f(x,xy)}{x^2}$. 
\end{bem}

\section{Constructing faces}
\begin{defi}\mbox{}
\begin{enumerate}
\item For a form $q \in \form_k$ and $p \in \proj^2(\R)$ with $q(p)=0$ we set 
\[ \mathbb{T}_p(q):=\{ x \in \proj^2 \colon \ \langle x, \nabla q(p) \rangle=0    \}; \]
the \emph{tangent space of $q$ in $p$}.
\item For a subset $J \subseteq H_k$ we define the vector space  
\[ J^{[2]}:= \spann\left(\{fg \colon \ f,g \in J\}\right)\subseteq H_{2k}. \]
\item 
For a face $F \subseteq \psd$ we set
\[ J_F := \{ q \in \form_2 \colon \ q^2 \in F \} \subseteq \form_2. \]
\end{enumerate}
\end{defi}

\begin{lemma}
For an arbitrary face $F \subseteq \psd$ the set $J_F$ is a vector space.
\end{lemma}
\begin{proof}
Obviously $J_F$ is closed under scalar multiplication. For $q_1,q_2 \in J_F$ we have  
$2(q_1^2+q_2^2)=(q_1+q_2)^2+(q_1-q_2)^2 \in F$, which proves that $q_1\pm q_2 \in J_F$. 
\end{proof}

Let $p \in \proj^2(\R)$ and $l$ a real line containing $p$. 

\begin{defi}
For $p,l$ as above we define following sets 
\begin{align*}
 &F_{p} := \left\{ f \in \psd \colon \ f(p)=0 \right\} \\
 &I_{p} := \left\{ f \in \form_{4}  \colon \   \ord_p(f)\geq 2 \right\} \\
 &F_{(p,l)}:= \left\{ f \in \psd \colon \ f(p)=0,\, \bu{f}(l_p)=0 \right\} \\
 &I_{(p,l)}:= \left\{ f \in \form_{4} \colon  \ \ord_p{f} \geq 2, \, \ord_{l_p}(\bu{f})\geq 2 \right\},  
\end{align*}
where all occurring blow-ups are taken in $p$. 
\end{defi} 

\begin{lemma}
\samepage
\label{fpexp}
Let $p \in \proj^2(\R)$ and $l$ a real line containing $p$. Then 
\begin{enumerate}
\item $F_{p}, F_{(p,l)} \subseteq \psd$ are faces.
\item $I_p \cap \psd = F_p$ and $I_{(p,l)} \cap \psd = F_{(p,l)}$. 
\item $J_F^{[2]}=\spann(F)$. 
\item The face $F_{p}$ is exposed.
\end{enumerate}
\end{lemma}

\begin{proof}\mbox{}
\begin{enumerate} 
\item 
Let $f,g \in F_{(p,l)}$ and $\lambda\geq 0$. Obviously we have $(f+g)(p)=0$, $(\lambda f)(p)=0$, as well as 
$\bu{(f+g)}(l_p)=\bu{f}(l_p)+\bu{g}(l_p)=0$ and $\bu{(\lambda f)}(l_p)=\lambda \bu{f}(l_p)=0$. 
This shows, that $F_p$ and $F_{(p,l)}$ are convex subcones of $\psd$. Now suppose $f,g \in \psd$ such that 
$f+g \in F_{(p,l)}$. Since $f,g$ are nonnegative and $(f+g)(p)=0$, we must have $f(p)=g(p)=0$. Also $\bu{f},\bu{g}$ are 
nonnegative polynomials, which shows, that $(\bu{f}+\bu{g})(l_p)=0$ implies $\bu{f}(l_p)=0$ and $\bu{g}(l_p)=0$.
\item
Suppose $f \in F_{(p,l)}$. It suffices to show that $f \in I_{(p,l)}$. Since $f$ is nonnegative, $p$ is a local minimum of 
$f$ 
and therefore $\ord_p(f) \geq 2$. The same argument shows that $\ord_{l_p}(\bu{f})$ is even and positive and therefore at 
least 2. 
\item 
Let $f,g \in J_{F}$. Then $4fg=(f+g)^2-(f-g)^2 \in \spann(F)$ and therefore $J_{F}^{[2]} \subseteq \spann(F)$. 
Now suppose $f \in \spann(F)$. Then $f=f_1-f_2$ with $f_1,f_2 \in F$. Since every nonnegative ternary quartic is a sum 
of squares, there exist $q_1,\ldots,q_r \in H_2$ such that $f=q_1^2+\dots +q_t^2 - q_{t+1}^2 -\dots - q_r^2$. Since 
$F$ is a face, we have $q_1^2,\ldots,q_r^2 \in F$ and therefore $q_1,\ldots,q_r \in J_F$. 
\item 
Evaluation in $p$ is a linear functional $L_p\colon \form_{3,4}\longto \R$, which is nonnegative on $\psd$. Obviously 
we have $\ker(L_p)\cap\psd = F_{p}$. 
\end{enumerate}   
\end{proof}

\begin{con}
\label{conv1}
We will write shortly $J_p$ and $J_{(p,l)}$ instead of $J_{F_p}$ and $J_{F_{(p,l)}}$ respectively. 
Similarly we write $J_f$ instead $J_{F_f}$ for $f \in \psd$. 
\end{con}

\begin{lemma}
\label{Jcomp}
Let $p,l$ be as above 
\begin{enumerate}
\item $J_{p}=\left\{ q \in \form_2 \colon \ q(p)=0 \right\}$,
\item $J_{(p,l)}=\left\{ q \in \form_2 \colon \ q(p)=0, \, l \subseteq \mathbb{T}_p(q) \right\}$,
\end{enumerate}
\end{lemma}
\begin{proof}
While the first assertion is obvious, the second one needs some explanation.
We choose affine coordinates in $p$ such that $l$ becomes the line $y=0$. 
For $q \in \form_2$ with $q(p)=0$ and $\nabla q(p)=(a,b)^t$ we have 
\begin{align*}
q^2 \in F_{(p,l)} &\iff q(p)=0 \text{ and }  \bu{q^2}(0,0)=0  \\
                  &\iff q(p)=0 \text{ and } (a+by)^2_{\mid y=0}=0\\
                  &\iff q(p)=0 \text{ and } a=0 \\
                  &\iff q(p)=0 \text{ and } \nabla q(p) \perp (1,0) \\
                  &\iff q(p)=0 \text{ and } l \subseteq \mathbb{T}_p(q).
\end{align*}  
\end{proof}

\begin{defi}
Let $\S= \{s_1,\ldots,s_n,(p_1,l_1),\ldots,(p_m,l_m) \}$ 
where $s_1,\ldots,s_n$, $p_1,\ldots,p_m \in \proj^2(\R)$ are  pairwise 
different and $l_1,\dots,l_m$ are real lines such that $p_i \in l_i$ for $i=1,\ldots,m$.  
We attach following sets to $\S$.  
\begin{align*}
 &F_{\S} := \bigcap_{i=1}^{n} F_{s_i} \cap \bigcap_{j=1}^{m} F_{(p_j,l_j)} \\
 &I_{\S} := \bigcap_{i=1}^{n} I_{s_i} \cap \bigcap_{j=1}^{m} I_{(p_j,l_j)} \\
\end{align*}
\end{defi}

\begin{cor}
For $\S$ as above we have 
\begin{enumerate}
\item $F_\S \subseteq \psd$ is a face. 
\item $I_\S \cap \psd = F_\S$
\end{enumerate}
\end{cor}

Similarly as in Convention \ref{conv1} we write shortly $J_\S$ instead of $J_{F_\S}$. 

\begin{bem}
The idea to prescribe certain configurations of zeros is due to Blekherman \cite{Blekh}. He used this approach to 
give an estimate of the dimension of certain faces of an \emph{arbitrary} cone of nonnegative forms. Moreover he was able 
to point out dimensional differences between faces of nonnegative forms and corresponding faces of the cone of sum of 
squares.  
\end{bem}

\section{Fullness}
In the last section we have seen, that $\spann(F_\S) \subseteq I_\S$. In general this inclusion is strict as the following
example shows. Since $\dim(I_\S)$ can be computed explicitly, we are interested in the case for which equality holds.
 
\begin{example}
Let $\S=\left\{(1\colon 0 \colon 0),(0 \colon 1 \colon 0),(1\colon 1 \colon 0) \right\}$. 
Each form $f \in F_{\S}$ meets the line $\Var(z)$ in three double points. Therefore Bezout's 
theorem implies that $f$ can be written as $f=zq$ with a cubic $q$. 
Since $f$ is positive, the gradient of $f$ must vanish on all real points of the line $\Var(z)$. 
Thus $q$ vanishes on $\Var(z)$ and hence $z \mid q$. 
Thus each form in $F_{\S}$ is divisible by $z^2$.  
On the other hand we have $g:=xy^2z-xyz^2 \in I_{\S}$ and therefore $\spann(F_{\S}) \subsetneq I_{\S}$. 
\end{example}

According to definition \ref{facedef} we have $\spann(F_\S)=I_\S$ if and only if there exists $f \in F_\S$ such that for all 
$g \in I_\S$ there exists an $\epsilon>0$ such that $f+\epsilon g \in \psd$. We will now find a criterion on $\S$, whether 
such an $f$ can be constructed. 

\begin{defi}
Let $f \in \form_4$ and $p \in \proj^2(\R)$ such that $\ord_p(f) \geq 2$. We choose affine coordinates in $p$ and define the 
discriminant $\Disk_f(y)$ of $f$ in $p$ as the polynomial $\Disk_f(y)=(f_3^2-4f_2 f_4)(1,y)$ where $f_2,f_3,f_4$ are the 
homogeneous components of $f$. Using $\bu{f}$ according to \ref{cbu} we conclude that $\Disk_f(y)$ is the 
discriminant with respect to $x$ of the quadratic polynomial $\bu{f}(y) \in \R[x]$. 
\end{defi}
 
\begin{lemma}
Let $p \in \proj^2(\R)$ and $l$ a real line through $p$. Suppose $f \in F_{(p,l)}$ and $g \in I_{(p,l)}$. If 
\begin{enumerate}
\item $\ord_p(f)=2$
\item $(\Disk_f)^{''}(0)\neq 0$
\end{enumerate}
then there exists a neighborhood $U\subseteq \proj^2(\R)$ of $p$ and $\epsilon>0$ such that $f+\epsilon g$ is nonnegative 
on $U$.  
\end{lemma}
\begin{proof}
We choose affine coordinates in $p$ so that $l$ is the line $y=0$. According to remark \ref{cbu} we have 
$l_p=0$.  We write $f_k=\sum_{i=0}^k 
a_{i,k-i} x^i y^{k-i}$ and $g_k=\sum_{i=0}^k b_{i,k-i} x^i y^{k-i}$ for the homogeneous components of $f$ and $g$ 
respectively. Since $f,g \in I_{(p,l)}$ we conclude 
\begin{align*}
a_{20}=a_{11}=a_{30}=0, \\
b_{20}=b_{11}=b_{30}=0.
\end{align*}
Thus $y \mid f_3, g_3$ and $y^2 \mid f_2,g_2$. Since $\bu{f}$ is nonnegative, $\Disk_f \leq 0$ and 
$y=0$ is a locale maximum of $D_f$. Fix $t>0$, then we have 
\[ \Disk_{f+tg}(y)=\Disk_f(y) + t^2P(y) + tQ(y) \]
for suitable $P,Q \in \R[y]$ with $y^2 \mid P(y),Q(y)$. Considering $(\Disk_f)^{''}(0)\neq 0$ 
we find $\delta,\epsilon_0 >0$ such that 
$\Disk_{f+t g}(y)\leq 0$ as long as $|y| \leq \delta$ and $0<t<\epsilon_0$.   
Since $f$ is nonnegative and $\ord_p(f)=2$ we must have $a_{02}>0$. We choose $0<\epsilon<\epsilon_0$ such that 
$a_{02}+\epsilon b_{02}=:c$ is positive, which implies that $f_2+\epsilon g_2=cy^2$ is a nonnegative form. 
We set $h:=(f+\epsilon g)$. 
Let $v:=(u,w) \in \R^2$ with $|v|=1$. We have  
\[ h(\lambda v)=c\lambda^2 w^2 + \lambda^3 R(v,\lambda) \]
for a suitable $R \in \R[v,\lambda]$. Hence there exists $\mu>0$ such that $h(\lambda v)\geq 0$ if $\lambda\leq \mu$ and 
$\left|\frac{w}{u}\right|\geq \delta$. Since $\Disk_h(y) \leq 0 $ for $y \leq \delta$ we conclude $h(x,y)\geq 0$ if 
$\left|\frac{y}{x}\right|\leq \delta$. 
Thus $h$ is nonnegative on $B_\mu(0)$.   
\end{proof}

\begin{theorem}
\label{mainthm}
Let $\S=\{s_1,\ldots,s_n,(p_1,l_1),\ldots,(p_m,l_m)\}$ be as above. If there exists $f \in F_\S$ such that  
\begin{enumerate} \item 
$\zero(f)=\{s_1,\ldots,s_n,p_1,\ldots,p_m\}$, 
\item 
$\inp_f(s_i)=\emptyset$ ,  
\item 
$\ord_{p_j}(f)=2$,
\item 
$(\Disk_f)^{''}(0)\neq 0$ in $p_j$,
\end{enumerate} 
for all $i=1,\ldots,n$ and all $j=1,\ldots,m$
then $\spann(F_\S)=I_\S$ and in particular $\dim I_\S=\dim F_\S$.
\end{theorem}
\begin{proof}
Let $g \in I_\S$ and $t \in \proj^2(\R)$. If $f(t)\neq 0$ there exists $\epsilon_t>0$ and an open  $U_t \subseteq 
\proj^2(\R)$ containing $t$ such that $f + \epsilon_t g$ is nonnegative on $U_t$. If $t \in \{p_1,\ldots,p_m\}$ the 
preceding lemma implies the  existence of an $\epsilon_t>0$ and an open
$U_t \subseteq \proj^2$ containing $t$ such that $f+\epsilon t g$ 
is nonnegative on $U_t$. Suppose now $t \in \{s_i,\ldots,s_n\}$. Since $\inp_f(t)=\emptyset$ 
the form $f_2$ (in $t$) is positive definite. 
Hence there exists $\epsilon_t$ and $U_t$ like above, such that $f+\epsilon_t g$ is nonnegative on $U_t$. Since 
$\proj^2(\R)$ is compact we 
can cover $\proj^2(\R)$ by finitely many of these $U_t$. Choosing $\epsilon$ minimal among all occurring $\epsilon_t$ shows 
that
$f+\epsilon g$ is nonnegative on $\proj^2(\R)$ and thus $g \in \spann(F_\S)$.     
\end{proof}

Since the above condition on the discriminant of $f$ is rather bulky, we provide a more algebraic respectively algorithmic 
version of above theorem in the following section. 

\section{Algebraic conditions}
Let $\S=\{s_1,\ldots,s_n,(p_1,l_1),\ldots,(p_m,l_m)\}$ be defined as in the preceding section. 
Since $J_\S^{[2] }=\spann(F_\S)$ we can 
work rather with $J_\S$ than with $F_\S$.

\begin{defi}
For a vector space $J\subseteq \form_2$ and $p \in \proj^2(\R)$ we define 
\begin{align*}
G_J(p) &= \{ \nabla q(p) \colon \ q \in J \} \subseteq \proj^2(\R) \\
E_J(p) &= \{ q \in J \colon \ \ord_p(q)\geq 2 \}
\end{align*}
\end{defi}

\pagebreak[3]

\begin{theorem}
\label{algcon}
Let $J \subseteq J_\S$. If 
\begin{enumerate}
\item $\zero(J)=\{s_1,\ldots,s_n,p_1,\ldots,p_m\}$
\item $\dim G_J(s_i) = 1$ 
\item $G_J(p_j) \neq \emptyset$ 
\item $\zero(E_J(p_j)) \cap l_j = \{p_j\}$ 
\end{enumerate}
for all $i=1,\ldots,n$ and all $j=1,\ldots,m$, then $F_\S$ is full in $I_\S$. 
Moreover for any basis $q_1,\ldots,q_k$ of $J$ the form $q_1^2+\dots+ q_k^2$ is an inner form of $F_\S$. 
\end{theorem}
\begin{proof}
Let $q_1,\ldots,q_k$ be a basis of $J$ and set $f:=q_1^2+\dots+q_k^2$. We show that $f$ satisfies all conditions of 
Theorem \ref{mainthm}. The first condition implies $\zero(f)=\{s_1,\ldots,s_n,p_1,\ldots,p_m\}$. Fix $s \in 
\{s_1,\ldots,s_n\}$ and choose affine coordinates in $s$.  The second condition implies $k\geq 2$ and we can assume 
that $\nabla q_1(s):=(a_1,b_1)^t$ and $\nabla q_2(s):=(a_2,b_2)^t$ are linearly independent. We may assume that 
$f=q_1^2+q_2^2$. We show that $\inp_f(s)=\emptyset$. We have 
\[ \bu{g}(0,y)=(a_1^2 + a_2^2)+ 2(a_1b_1 + a_2 b_2)y + (b_1^2+b_2^2)y^2. \]
This polynomial has discriminant 
\[  4(a_1b_1+a_2b_2)^2 - 4(a_1^2+a_2^2)(b_1^2+b_2^2)=-(a_1b_2-a_2b_1)^2= - 
\left| \begin{array}{cc}
a_1& b_1 \\
a_2 & b_2 
\end{array} 
\right|^2 < 0. \] Moreover $(b_1^2+b_2^2)\neq 0$ and thus $\inp_f(s)=\emptyset$. 
Fix $(p,l) \in \{(p_1,l_1)\ldots,(p_m,l_m)\}$ and choose affine coordinates in $p$ such that $l$ is the line $\Var(y)$.
Since $\bu{f}(0,0)=0$ we must have $\nabla q_i(p)=(0,b_i)$ for $i=1,\ldots,k$ and suitable $b_i \in \R$.
Condition (iii) implies that there exists $i \in \{1,\ldots,k\}$ such that $b_i \neq 0$ and hence $\ord_p(f) = 2$. 
We write $q_i(x,y)=b_i y + c_i x^2 + d_i xy + e_i y^2$ with $c_i,d_i,e_i \in \R$ and $i=1,\ldots,k$. Condition (iv) implies 
that the vectors $b:=(b_1,\ldots,b_k)$ and $c:=(c_i,\ldots,c_k)$ are linearly independent. On the other hand we have 
\[ \Disk_f(y)=(f_3^2-4f_2f_4)(1,y) \] and thus 
 \begin{align*}
(\Disk_f)''(0) &= 8 \left( \sum_{i=1}^k b_i c_i \right)^2 - 8 \left(\sum_{i=1}^k b_i^2\right) \left(\sum_{i=1}^k c_i^2 
\right) \allowdisplaybreaks[1] \\ 
 &= \sum_{i,j=1}^k \left( b_i b_j c_i c_j - b_i^2 c_j^2 \right)\allowdisplaybreaks[1] \\
 &=\sum_{i,j=1}^k \left( b_i c_j(c_i b_j - b_i c_j) \right)  \\
 &= \sum_{i<j} \left( c_i b_j - b_i c_j\right)^2. 
\end{align*} 
The last sum runs over all $2\times 2$--Minors of the Matrix $(b,c)$. Therefore $(\Disk_f)''(0)\neq 0$. This shows, that 
$\spann(F_\S)=I_\S$ and $f$ is an inner form of $F_\S$. 
\end{proof}

\begin{bem}
Theorem \ref{mainthm} as well as theorem \ref{algcon} can be generalized to quartics with an arbitrary finite number of 
variables. 
\end{bem}

\section{Equivalence}
Let $\G:=\GL$ denote the real projective linear group on $\proj^2$. 
If we identify positive multiples of forms in $\form_4$ (resp. $\form_2$) the group $\G$ operates on $\form_4$ (resp. on  
$\form_2$) by
\[ (\sigma \cdot f)(x):= f(\sigma^{-1}(x)), \quad f \in \form_4 \text{ (resp. }  f\in \form_2), \sigma \in \G.  \]
For $\S=\{s_1,\ldots,s_n,(p_1,l_1),\ldots,(p_m,l_m)\}$ as before and $\sigma \in \G$ we define 
\[ \sigma(\S)=\{\sigma(s_1),\ldots,\sigma(s_n),(\sigma(p_1),\sigma(l_1)),\ldots,(\sigma(p_m),\sigma(l_m) \}. \]

\begin{lemma}
For $\S$ as above, $\sigma \in \G$ and $F \subseteq \psd$ a face we have 
\begin{enumerate}
\item $\sigma(F)$ is a face of $\psd$ and $\dim(F)=\dim(\sigma(F))$. 
\item $\sigma(F_\S)=F_{\sigma(\S)}$, $\sigma(I_\S)=I_{\sigma(\S)}$ and $\sigma(J_\S)=J_{\sigma(\S)}$. 
\item $F$  exposed implies $\sigma(F)$ exposed. 
\end{enumerate}
\end{lemma}
\begin{proof}
The first assertion is obvious, since $\sigma$ is linear and invertible and satisfies $\sigma(\psd)=\psd$. 
The second assertion is obvious if $m=0$, thus it is enough to show the claim for $\S=\{(p,l)\}$.
First we notice that $\ord_p(f)=\ord_{\sigma(p)}(\sigma(f))$ for all $f \in H_4$.
We choose affine
representatives of $p$ and $\sigma$, which we will also denote by $p$ and $\sigma$ respectively. We write $p'=\sigma(p) \in 
\R^3$ and choose $q$ such that $l=\spann(p,q)$. For $q':=\sigma(p')$ we have $\sigma(l)=\spann(p',q')$. 
Let further be $r \notin l$. Using $r$ and $\sigma(r)$ respectively we construct affine coordinates $x,y$ in $p$ such that 
$l$ becomes the line $\Var(y)$ and affine coordinates $x',y'$ in $p'$ such that $\sigma(l)$ becomes the line $\Var{y'}$ 
respectively. In this coordinates we compute $f$ and $\sigma(f)$ respectively and obtain
\begin{align*}
 f(x,y) &:= f(p+x(p-q)+yr) \\
 (\sigma f)(x',y') &:= (\sigma f)(p'+x'(p'-q') + y' \sigma(r)) \\
                &= f(\sigma^{-1}(p'+x'(p'-q')+y'\sigma(r))) \\
                &=f(p+x'(p-q)+y'r) 
\end{align*}
which makes apparent, that $\ord_{l_p}(\bu{f})=\ord_{\sigma(l)_{\sigma(p)}}(\bu{\sigma(f)})$. 
\item 
Let $L$ be a linear functional which defines $F$. If we fix an affine representative for $\sigma$ we have a
linear map $\varphi \colon \form_{3,4} \longto \form_{3,4}$ given by $f \mapsto \sigma f$.  
Then $L \circ \varphi$ is a linear functional, which is nonnegative on $\psd$. For $f \in \psd$ we have 
\[
f \in \ker(L\circ \varphi) \iff \sigma f \in \ker(L).  
\]
\end{proof}

\begin{defi}
Let $F_1,F_2 \subseteq \psd$ be faces. We call $F_1$ equivalent to $F_2$ if there exists $\sigma \in \G$ such that
$\sigma(F_1)=F_2$. In this case we write $F_1 \sim F_2$. Evidently $\sim$ is an equivalence relation. Denote $\mathcal{F}$ 
the quotient set.  For $[F] \in \mathcal{F}$ we define $\dim([F])=\dim(F)$, due to the preceding lemma this definition does 
not depend on the choice of a particular representative. 
\end{defi}

Henceforth we will restrict our considerations on equivalence classes of faces. 

\begin{defi}
For $k\in \N$ we set 
\[ \V_k := \left\{ (L_1,\ldots,L_k) \colon \ L_i \subset \proj^2(\R)
 \text{ subspace, } -1 \leq \dim L_1 \leq \ldots \leq \dim L_k 
\leq 2 \right\}. \]
With the convention $g\cdot \emptyset=\emptyset$ for $g \in \G$ the group $\G$ operates on $\V_k$ entrywise.
\end{defi}

The next 
lemma is a well-known fact from linear algebra.

\begin{lemma}
\label{orbits}
Following sets form (complete) $\G$--orbits in $\V_4$. 
\begin{align*}
& \left\{ (p_1,p_2,p_3,p_4) \colon \ p_i \textup{ points, no 3 of them collinear } \right\} \\
& \left\{ (p_1,p_2,l_1,l_2) \colon \ p_i \textup{ points, } l_i \textup{ lines, } p_i \in l_i,\; p_i \notin l_j, \, i\neq j 
\right\} \\
& \left\{ (p_1,p_2,p_3,l) \colon \ p_1, p_2, p_3 \textup{ non-collinear points, } l \textup{ a line, }p_1 \in l, \, p_2,p_3 
\notin l \right\} 
\end{align*}
In particular $\G$ operates transitively on each of these sets projected on a $\V_k$, with $1\leq k\leq 4$.  
\end{lemma} 

Obviously $\V_4$ splits in far more orbits, but only the mentioned ones are of interest for us.

\begin{bem}
According to above orbits, it is useful to keep following pictures in mind,

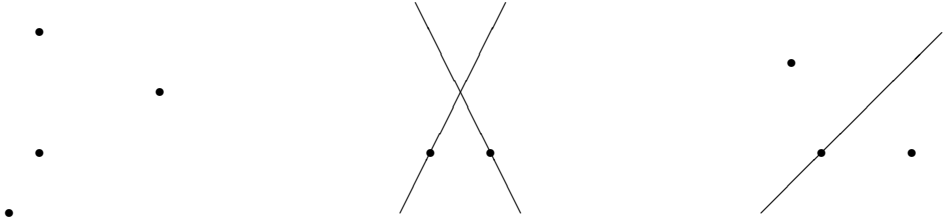
\setlength{\unitlength}{0.4cm}
\begin{picture}(32,10)(1,-2)
\put(2,0){\circle*{0.3}}
\put(3,2){\circle*{0.3}}
\put(7,4){\circle*{0.3}}
\put(3,6){\circle*{0.3}}

\put(15,0){\line(1,2){3.5}}
\put(19,0){\line(-1,2){3.5}}
\put(16,2){\circle*{0.3}}
\put(18,2){\circle*{0.3}}

\put(27,0){\line(1,1){6}}
\put(29,2){\circle*{0.3}}
\put(32,2){\circle*{0.3}}
\put(28,5){\circle*{0.3}}
\end{picture}

which display the above configurations $\S$, on which $\G$ operates transitively. 
\end{bem}

\begin{lemma}
\label{reduction}
Let $f \in \psd$. 
\begin{enumerate}
\item Suppose $f$ has three zeros on the real line $l$, then $f$ is a product 
of $l^2$ and a nonnegative quadratic form.  
\item Let $p,q \in \proj^2(\R)$ distinct and $l=p \vee q$. Suppose $f(p)=f(q)=0$. 
If $l_p \in \inp_p(f)$ then $f$ vanishes on the entire line $l$. 
In particular $f$ is a product of $l^2$ and a nonnegative quadratic form. 
\end{enumerate}
\end{lemma}  
\begin{proof}
Since $f \in \psd$ we can write $f=g_1^2+\dots+g_k^2$ for suitable $g_i \in \form_{2}$. 
For the first statement we note, that $g_1,\ldots,g_k$ vanish in each of the three points on $l$. 
Then Bezout's Theorem implies that each $g_i$ is divisible by $l$ and the claim follows. 
For the second statement we note that $\nabla g_i(p)$ and $\nabla l$ are linearly 
dependent for all $i=1,\ldots,k$. Therefore $g_i$ and $l$ have a at least a double intersection at $p$. Again Bezout's 
Theorem implies that each $g_i$ is divisible by $l$.
\end{proof}

\begin{cor}
\label{type}
Let $F \subseteq \psd$ be a face, then exactly one of the following statements holds 
\begin{enumerate}
\item[$(\mathcal{A})$] $\zero(F)$ contains a real line. 
\item[$(\mathcal{B})$] There exists $\S$ such that $F=F_\S$ and $\zero(F)$ contains no real line.  
\item[$(\mathcal{C})$] $F$ cannot be expressed as $F_\S$ for a suitable $\S$ and $\zero(F)$ contains no real line.   
\end{enumerate}
Moreover if $G \in [F]$ the same statement holds for $G$. 
\end{cor}

This gives us a partition of the set of all faces of $\psd$ into three parts, which we denote by 
$\typeA$, $\typeB$ and $\typeC$ according to the numeration in corollary \ref{type}. This partition is 
compatible with the equivalence relation on $\psd$.    

In the following sections we will give a complete list of faces of each 
type. In the last section we will eventually discuss inclusions between faces.

\section{Faces \texorpdfstring{in $\typeA$}{of type A}}
Let $F \in \typeA$ and $f$ be an inner form of $F$. 
The next theorem shows that faces in $\typeA$ are well-understood.  

\begin{theorem}[cf. \cite{Bar}, ch. II.12]
\label{psdmatrix}
Let $l$ be a line in $\proj^{2}(\R)$. Then there exists a bijection 
\begin{align*}
\left\{ F_f \colon \ f \in \psd, \ \,  l^2 \mid f \right\} &\longto  \left\{ U \subset \proj^{2}(\R) \colon \ \textup{$U$ 
subspace of $\proj^2(\R)$} \,  \right\} \\
                F_f &\longmapsto \zero\left(\frac{f}{l^2}\right) 
\end{align*}
\end{theorem}
\begin{proof}
Since $f$ is nonnegative, $f/l^2$ is a positive semidefinite quadratic 
form. Therefore the facial structure
of the cone $C:=\{ f \in \psd \colon \   l^2 \mid f \}$ can be identified with the facial structure 
of symmetric positive definite $3 \times 3$--matrices over $\R$. 
In \cite{Bar} it is shown, that faces of this cone can be parameterized by subspaces of $\proj^2(\R)$ in terms of 
simultaneous kernels.  
\end{proof}

Therefore each (unordered) pair $(l,U) \in \V_2$, where $l$ is a line, corresponds 
to a face $F \in \typeA$. Thus we can regard $\typeA$ as a subset of $\V_2$. The $\G$--operation on $\V_2$ restricted to 
$\typeA$ coincides with above operation of $\G$ on $\psd$.  
Let $l,k$ be distinct real lines and $p,q \in \proj^{2}(\R)$ with $p \in l$ and $q \notin l$. Then 
Lemma \ref{orbits} implies that $\typeA\subset \V_2$ consists of six equivalence classes, for which representatives
are given as follows 
\[
F_{\Li_0}:=(l,\emptyset),\; F_{\Li_1} :=(l,q),\; F_{\Li_2} :=(l,p)
\]
\[ 
F_{\Li_3}:=(l,k),\; F_{\Li_4} :=(l,l), \; F_0 =(l,\proj^{2}).
\]

\begin{cor}
The set $\typeA$ splits in six equivalence classes, which are given by $F_{\Li_0}$, $F_{\Li_1}$, $F_{\Li_2}$, $F_{\Li_3}$, 
$F_{\Li_4}$ and $F_0$. Every face in $\typeA$ is exposed. Furthermore we have 
\begin{center}
\begin{tabular}{l|c|c|c|c|c|c}
\emph{Face} $F$ & $F_{\Li_0}$ & $F_{\Li_1}$ & $F_{\Li_2}$ & $F_{\Li_3}$ &$ F_{\Li_4 }$& $F_0$ \\ 
\hline 
$\dim F$ &  6 & 3 & 3 & 1  & 1 & 0  \\
$\dim J_F$ & 3 & 2& 2& 1& 1& 0 \\
\end{tabular}
\end{center} 
\end{cor}
\begin{proof}
We first show the statement concerning given dimensions. 
Let $F \in \typeA$ be given by $(l,U) \in \typeA \subseteq \V_2$. 
Theorem \ref{psdmatrix} implies that $\dim(F) =\dim \{ q \in \form_{2} \colon \ q(U)=0 \}$, which can be computed easily. 
Moreover we have $J_\S=\{ lm \colon \ m \in \form_{1}, \, m(U)=0 \}$, which implies $\dim J_\S=2-\dim U$.  
It remains to show that all these faces are exposed. This is clear however, since every face in $\typeA$ can be 
written as finite intersection of some $F_{p_i}$ for suitable $p_i \in \proj^2(\R)$. 
\end{proof}  

\section{Faces \texorpdfstring{in $\typeB$}{of type B}}
For the rest of this paper we set $e_1=(1\colon 0\colon 0)$, $e_2:=(0\colon 1 \colon 0)$, $e_3=(0 \colon 0 \colon 1)$ and 
$e_4=(1 \colon 1 \colon 1)$. By definition each face $F \in \typeB$ is of the form $F_\S$ for 
$\S=\{s_1,\ldots,s_n,(p_1,l_1),\ldots,(p_m,l_m) \}$ and suitable $s_i,p_i,l_i$.  
Since $F \notin \typeA$ we may assume that no three $s_i$ lie on one line and $l_i \cap 
\{s_1,\ldots,s_n,p_1,\ldots,p_m\}=\{p_i\}$ for $i=1,\ldots,m$. 

Suppose now $2m+n\geq 5$. Let $f \in F_S$ and write $f=q_1^2+\dots+q_k^2$. Then 
$q_1,\ldots,q_k$ meet in at least $2m+n$ points (counted with multiplicity).    
Since $F_S \notin \typeA$, Bezout's theorem implies that all $q_1,\ldots,q_k$ are multiples of an indefinite 
regular quadratic form $q$. Since all regular indefinite quadratic forms arise from 
$g:=x^2-y^2+z^2$ by coordinate change, we have $F_\S \sim F_{q^2}$. We denote $F_{q^2}$ by $F_{\Qu}$.

Hence we can assume from now on that $2m+n<4$. 
We first consider the case $m=0$. We set 
\[ \S_1:=\{e_1\},\, \S_2:=\{e_1,e_2\},\,  \S_3:=\{e_1,e_2,e_3\}, \, \S_4:=\{e_1,e_2,e_3,e_4\} \]

\begin{lemma}
For $i=1,\ldots,4$ the face $F_{\S_i}$ is full in $I_{\S_i}$. 
\end{lemma} 
\begin{proof}
Using Lemma \ref{Jcomp} we obtain 
\begin{align*}
J_{\S_1}&=\spann(y^2,z^2,xy,xz,yz), \\
J_{\S_2}&=\spann(z^2,xy,xz,yz), \\
J_{\S_3}&=\spann(xy,xz,yz), \\
J_{\S_4}&=\spann(xy-yz,yz-xz), \\
\end{align*}
which makes apparent that conditions (i) and (ii) of theorem \ref{mainthm} are satisfied.
\end{proof}

Note that lemma \ref{orbits} implies that each $\S=\{s_1,\ldots,s_n\}$ with $n \leq 4$ and no three $s_i$ on a line is 
$\G$--equivalent to $\S_n$. 

Next we consider the case $1 \leq m \leq 2$. We set 
\begin{align*}
\T_1 &:= \{(e_1,e_1 \vee e_3) \} \\
\T_2 &:=\{e_2, (e_1, e_1 \vee e_3) \} \\
\T_3 &:=\{(e_1,e_1 \vee e_3),(e_2,e_2 \vee e_3) \} \\
\T_4 &:=\{e_2,e_3,(e_1 \vee e_4) \}
\end{align*}

\begin{lemma}
For $i=1,\ldots,4$ the face $F_{\T_i}$ is full in $I_{\T_i}$. 
\end{lemma}
\begin{proof}
We compute
\begin{align*}
J_{\T_1} &=\spann(y^2,z^2,xy,yz) \\
J_{\T_2} &=\spann(z^2,xy,yz) \\
J_{\T_3} &=\spann(z^2,xy) \\
J_{\T_4} &=\spann(yz,x(y-z)).
\end{align*}
It is enough to check conditions (i)-(iv) of Theorem \ref{mainthm}. For instance consider $\T_2$. We have 
$\zero(z^2,xy,yz)=\{e_1,e_2\}$. Moreover $\nabla (xy)(e_2)$ and $\nabla (yz)(e_2)$ are linearly independent. 
Thus conditions (i) and (ii) are satisfied. Since $\nabla (xy)(e_1) \neq 0$ also condition (iii) is satisfied. 
For condition (iv) consider the form $z^2$, which satisfies $\zero(z^2)\cap (e_1\vee e_3) = \{e_1\}$.
For the remaining cases one can argue similarly.
\end{proof}

We summarize the last results in following theorem 

\begin{theorem}
The set $\typeB$ splits in 10 equivalence classes with respective dimension as follows 
\begin{center}
\begin{tabular}{l|c|c|c|c|c|c|c|c|c|c}
\emph{Face} $F$ & $F_{\emptyset}$ & $F_{\S_1}$ & $F_{\S_2}$ & $F_{\S_3}$ & $F_{\S_4}$ & $F_{\T_1}$ & 
$F_{\T_2}$ & $F_{\T_3}$ & 
$F_{\T_4}$ & $F_\Qu$ \\ \hline 
$\dim F$ & 15 & 12 & 9 & 6 & 3& 9  & 6 &  3 & 3 & 1 \\
$\dim J_F$& 6 &   5  & 4 &  3 & 2& 4  & 3  &  2 & 2  & 1
\end{tabular}
\end{center}
Furthermore $F_{\S_1},F_{\S_2},F_{\S_3},F_{\S_4}$ and $F_{\Qu}$ are exposed faces.  
\end{theorem}
\begin{proof}
Let $\S=\{s_1,\ldots,s_n,(p_1,l_1),\ldots,(p_m,l_m) \}$ such that $F_\S \in \typeB$. If $n+2m \geq 5$ then $F_\S\sim F_\Qu$. 
If $1\leq n+2m \leq 4$ then Lemma \ref{orbits} implies the existence of $\sigma \in \G$ and $i \in \{1,2,3,4\}$ such that 
$F_{\sigma(S)}=F_{S_i}$ or $F_{\sigma(S)}=F_{T_i}$. Thus $\typeB$ splits in at most 10 equivalence classes. 
The assertions concerning dimensions are easily verified by computing a basis of $I_F$ and $J_F$ respectively. Since 
equivalence preserves dimensions of faces, it remains to show that above faces of same dimension are pairwise 
inequivalent. 
The form $f:=z^4+y^4 \in F_{\T_1}$  has only one real zero and hence $\sigma(f) \notin 
F_{\S_2}$ for all $\sigma \in \G$, which implies $F_{\S_2} \not\sim F_{\T_2}$. Similarly one shows $F_{\S_3} \not\sim 
F_{\T_2}$ 
and $F_{\S_4} \not\sim F_{\T_3},F_{\T_4}$. The form $g:=z^4+x^2y^4 \in F_{\T_3}$ has only 2 real zeros, which implies 
$F_{\T_4} \not\sim F_{\T_3}$. 
Lemma \ref{fpexp} implies that the faces $F_{\S_i}$ for $i=1,2,3,4$ are exposed. Since $F_\Qu$ is equivalent to $F_\S$ 
where $\S$ contains 5 pairwise different real points on an irreducible indefinite conic, also $F_\Qu$ is exposed.   
\end{proof}

\section{Faces \texorpdfstring{in $\typeC$}{of type C}}
Last we will construct all remaining faces. Suppose $F \in \typeC$ and let $f \in F$ an inner form of $F$. 
We have seen that if $\zero(f)$ is an infinite set then either $F \sim F_Q$ or $F \in \typeA$. Hence we may assume
that $\zero(f)$ is finite. If $t \in \zero(f)$ then either $\ord_t(f)=2$ or $\ord_t(f)=4$. 

\subsection{A "degenerate" face}
Suppose $\ord_t(f)=4$. If $f$ had another real zero $s$ then $f$ would intersect the line $s\vee t$ at least 5 times and 
hence $(s\vee t) \subseteq \zero(f)$; a contradiction. By coordinate change we may assume that $\zero(f)=\{e_1\}$. 
Thus the variable $x$ cannot occur in $f$ and therefore $f \in \R[y,z]$. Moreover $f$ has no real zeros on the line 
$\Var(x) \cong \proj^1$. Since $f \in F$ is an inner form of $F$, we must have $\ord_t(g)\geq 4$ for all $g \in F$. 
Conversely let $g \in \form_4$ such that $\ord_t(g)\geq 4$ (i.e. $g\in \R[y,z]$). Since $f$ is strictly positive on 
$\proj^1(\R)$ there exists $\epsilon>0$ such that $f+\epsilon g$ is nonnegative on the real points of $\Var(x)$. Since $x$ 
does not occur in $f +\epsilon g$, we conclude that $f+\epsilon g$ is globally nonnegative. Thus we have 
\[ \spann(F)=\{ g \in \form_4 \colon \ \ord_{e_1}(g)\geq 4 \}= \R[y,z]_4.  \]
We denote the above face by $F_\D$ and conclude 
\[\dim(F_\D)=\dim(\R[y,z]_4)=\binom{5}{4}=5.\]

\subsection{The remaining faces}
From now on we assume that $f$ has zeros only of order 2. Let $\{s_1,\ldots,s_n\}$ be the set of real zeros of $f$ with 
$\inp_f(s_i)=\emptyset$ for $i=1,\ldots,n$ and $p_1,\ldots,p_m$ the remaining real zeros. Let $l_1,\ldots,l_m \subset 
\proj^2$ 
lines such that $(l_i)_{p_i} \in \inp_f(p_i)$ for $i=1,\ldots,m$. As usual we write 
$\S=\{s_1,\ldots,s_n,(p_1,l_1),\ldots,(p_m,l_m)\}$. Since $F \notin \typeB$ we must have $F \subsetneq F_\S$. Instead of 
viewing at $F_f$ we will work with $J_f=J_F$, which carries all information of $F$. Since $F \subsetneq F_\S$ we have 
$J_F \subsetneq J_\S$. Hence we are interested in vector subspaces $J \subsetneq J_\S$ such that $J_f$ for $f$ an inner 
form of $F \in \typeC$.

\begin{lemma}
\samepage
\label{nfdcrit}
Let $f$ and $\S$ be as above and $J=J_f$. Then
\begin{enumerate}
\item $\zero(J)=\{s_1,\ldots,s_n,p_1,\ldots,p_m\}$
\item $\dim G_J(s_i) \geq 1$ 
\item $G_J(p_j) \neq \emptyset$ 
\end{enumerate}
for all $1\leq i \leq n$ and $1\leq j \leq m$. Moreover we have $\zero(E_J(p_j)) \cap l_j \neq \{p_j\}$ for at least 
one $1 \leq j \leq m$.   
\end{lemma}
\begin{proof}
Since $\zero(f)=\{s_1,\ldots,s_n,p_1,\ldots,p_m\}$ and $J\subset J_\S$ the first assertion must hold. 
Fix $1\leq i \leq n$ and choose affine coordinates in $s_i$. Suppose $G_J(s_i) < 1$, then there exists 
$v \in \R$ such that $\nabla q (s_i)$ is a multiple of $v$ for each $q \in J$. Hence 
\[ \bu{f}(0,y)= \lambda v^t (1,y)^t \]
for some $\lambda \in \R$. But this polynomial has a real zero (for $v=(1,0)^t$, this zero is at infinity) and 
hence $\inp_f(s_i) \neq \emptyset$, which is a contradiction. This proves the second assertion. 
If $G_J(p_j)=\emptyset$ for a $1\leq j \leq m$, then every $q \in J$ has a singularity in $p_j$. Hence $\ord_{p_j}(f)=4$, 
which is also a contradiction. Last suppose $\zero(E_J(p_j)) \cap l_j \neq \{p_j\}$ for all $j$. Then theorem \ref{algcon}
implies that $F=F_\S$, which is a contradiction.     
\end{proof}

\begin{lemma}
\label{satcrit1}
Let $F \subseteq \psd$ be a face and let $J$ be an arbitrary subspace of $J_F$ with basis $q_1,\ldots,q_k$. If 
\[ q_1^2+ \dots +q_k^2- \epsilon h^2 \]
is indefinite for all $h \in J_F\setminus{J}$ and all $\epsilon>0$. Then $J=J_G$ for a face $G \subseteq F$.    
\end{lemma}
\begin{proof}
We set $g:=q_1^2+\dots+q_k^2$ and claim $J=J_g$. Since $g \in F$ we have $F_g \subseteq F$. 
By definition we have $q_1,\ldots,q_k \in J_g$ and thus $J \subseteq J_g$. Let now $h \in J_g \subseteq J_F$. Then there 
exists $\epsilon>0$ such that $g-\epsilon h^2$ is nonnegative. Thus we must have $h \in J$.  
\end{proof}

\begin{cor}
\label{satcrit2}
Let $F \subseteq \psd$ a face and $J$ a subspace of $J_F$ of codimension 1 with basis $q_1,\ldots,q_k$. If there exists 
$h\in J_F\setminus J$ such that 
\[ q_1^2+\dots+q_k^2 - \epsilon h^2 \]
is indefinite for all $\epsilon >0$ then there exists a face $G\subsetneq F$ such that $J=J_G$. 
\end{cor}
\begin{proof}
We apply the preceding lemma.
We set $g:=q_1^2+\dots+q_k^2$.  
Suppose there exists $p \in J_F\setminus{J}$ and $\epsilon>0$ such that $g-\epsilon p^2$ is nonnegative. 
This implies $p \in J_g$. On the other hand we have $J \subseteq J_g$. Since $J$ has codimension 1 in $J_F$ 
we must have $J_g=J_F$ and in particular $h \in J_g$, which is a contradiction.  
\end{proof}

Lemma \ref{nfdcrit} provides a method to find all faces in $\typeC$.  We proceed in the following way. 
First we fix a configuration of $\S$. Next we will compute all subspaces $J$ of $J_\S$ which satisfy the conditions of lemma 
\ref{nfdcrit}. Then we will cancel out redundancies regarding equivalent faces. Last we will use Corollary \ref{satcrit2} in 
order to show that these subspaces arise in fact from faces.  
Since there are only 8 inequivalent configurations of $\S$ this procedure will eventually give us all faces of 
$\typeC$. 

We will now distinguish cases according to different types of $\S$. 
If $m=0$ then theorem \ref{mainthm} implies that $f$ is an inner Form of $F_\S$ and hence $F =F_\S \in \typeB$. Thus 
we may assume $m>0$. 

The first interesting case is $m=1$ and $n=0$. We may assume $\S=\T_1=\{(e_1,e_1 \vee e_2)\}$. 
 A basis of $J_{\T_1}$ is given by 
\[
\left\{ xy, y^2, yz, z^2\right\}. 
\]
We are interested in subspaces $J \subset J_S$ satisfying the conditions of lemma \ref{nfdcrit}.
We write $q_i:=A_i xy + B_i y^2 + C_i yz + D_i z^2$ for the basis elements of $J$, where $1\leq i \leq \dim J$. 
Condition (iii) implies that $A_i \neq 0$ for one $i$. Therefore we may assume $A_1=1$ and $A_i=0$ for $i \neq 
1$. Thus $E=\{ q \in J \colon \ \nabla q(e_1)=0 \}$ is spanned by the $q_i$ with $i=2,\dots,\dim J$.
Since we demand $\zero(E) \cap (e_1 \! \vee \! e_3) \neq \{e_1\}$ we also have $D_i=0$ for 
$i \neq 1$. We distinguish cases according to the dimension of $J$. \\ 

Case 1: $\dim J =3$. \\
In this case we may assume that $J$ has basis 
\[ \{xy + D z^2,yz, y^2 \}, \quad D \in \R. \]
We have $\zero(J) = \{e_1\}$ if and only if $D \neq 0$. 

The matrix $M=\operatorname{diag}(\frac{1}{D},1,1)$ induces an element of $\stab(\T_1) \subset \G$, sending
$J$ to $\spann\{ xy+z^2,yz,y^2\}$. Therefore we assume without loss of generality $D=1$. 
Last we check the condition of Corollary \ref{satcrit2}.
For $\epsilon>0$ we set 
\[ g:= (xy + z^2)^2 + y^2z^2+ y^4 - \epsilon z^4. \]
For $t \in \R$ we have
\[ g(1,-t^2, t)=  t^8 + t^6 - \epsilon t^4, \]
which is negative for all $\epsilon>0$ and sufficiently small $t$.   
For $f=(xy+z^2)^2+y^2z^2+y^4$ we denote $F_f$ and $J_f$ by $F_{\T_1}^*$ and $J_{\T_1}^*$ respectively. \\ 
\pagebreak[2]

Case 2: $\dim J =2$ \\
We start with two basis vectors $q_1=xy+ B_1y^2+C_1yz+ D_1z^2$ and $q_2= B_2 y^2 + C_2 yz $. Again we must have 
$D_1 \neq 0$, otherwise $q_1$ and $q_2$ vanish on the line $\Var(y)$. If $D_1\neq 0$ the line $\Var(y)$ intersects $q_1$ of 
order two in $e_1$.  Suppose $C_2 \neq 0$, then $q_2$ factors in two distinct lines. Since $\Var(y)$ intersects $q_1$ of 
order two in $e_1$, the other line has only intersects $q_1$ of order one in $e_1$, and therefore there exists another real 
intersection with $q_1$, which is a contradiction to $\zero(J)=\{e_1\}$. Thus we must have $C_2=0$ and we may assume  
$q_1=xz+C yz + D z^2$ and $q_2=y^2$ with $C,D \in \R$, $D \neq 0$. 

The matrix 
\[ 
 M= \left( \begin{array}{ccc}
1& 0& -C \\
0 & 1/D &0 \\
0 & 0 & 1\\
\end{array}
 \right)
\]
induces an element of $\stab(\T_1) \subset \G$ sending $J$ to $\spann\{xy+z^2,y^2\}$. Therefore we can assume 
$D=1,C=0$. 
 
By construction $J$ satisfies condition (i),(ii) and (iii) of lemma \ref{nfdcrit}. Since $J 
\subset J_{\T_1}^*$ we can apply Corollary \ref{satcrit2} and set
\[ g:= (xy+z^2)^2+y^4 - \epsilon y^2z^2, \quad \epsilon >0. \]
For $t \in \R$ we have 
\[
 g(1,-t^2,t) = t^8 -\epsilon t^6, 
 \]
which is negative for all $\epsilon>0$ and sufficiently small $t$. 
For $f=(xy+z^2)^2+y^4$ we denote $F_f$ and $J_f$ by $F_{\T_1}^{**}$ and $J_{\T_1}^{**}$ respectively. \\ 

Case 3: $\dim J=1$ \\
In this case $J$ is spanned by a single element $q$. By lemma \ref{nfdcrit} we must have $\zero(q)=\{e_1\}$ on the other 
hand we demand $\nabla q(e_1) \neq 0$. Such a $q \in \form_{2}$ cannot exist.   \\ \\
Thus up to know we have classified all faces $F_f$ which have an inner form $f$ with exactly one real zero.

We will now consider the case $m=1$ and $n=1$. It is enough to consider $\T_2=\{e_2,(e_1,e_1\vee e_3)\}$.
A basis of $J_{\T_2}$ is given by 
\[
\left\{ xy,yz,z^2  \right\}.
\]
We assume that a subspace $J \subset J_{\T_2}$ has basis $q_i=A_ixy + B_i yz+C_iz^2$ for $1\leq i\leq \dim J$.  
We distinguish cases according to the dimension of $J$. \\ 

Case 1: $\dim J=2$ \\
We first look at $e_1$. 
Condition (iii) of lemma \ref{nfdcrit} implies that the monom $xy$ has to occur in one of the $q_i$.  
Hence we can assume, that $A_1=1$ and $A_2=0$. Then $E=\{ q \in J \colon \ \nabla q(e_1)=0 \}$ is spanned by $q_2$. 
In  particular we have $\zero(E) \supset (e_1 \! \vee \! e_2 )$, thus the last condition holds independently of the choice 
of the remaining coefficients. 
For $e_2$ condition (ii) of lemma \ref{nfdcrit} implies $B_2 \neq 0$. Thus we can assume $B_2=1$ and $B_1=0$.  
Last we consider $\zero(J)=\zero(xy+C_1z^2,yz+C_2z^2)$. 
The form $q_2$ factors in the linear forms $l_1=z$ and $l_2=y+C_2z$. Then $e_1,e_2 \in \zero(l_1)$ and $e_1 \in \zero(l_2)$. 
Since $l_1\neq l_2$  the forms $q_1$ and $q_2$ intersect transversely in $e_2$. Since we demand
$\zero(q_1,q_2)=\{e_1,e_2\}$ the forms $q_1$ and $q_2$ have to be coprime and must intersect of order 3 in $e_1$. 
Thus $l_2$ must intersect $q_2$ of order~2 in $e_1$, which implies $C_2=0$. On the other hand we must demand $C_1\neq0$ 
otherwise $y \mid q_1$ and $y \mid q_2$.    
Thus we can assume that a basis of $J$ has the form $q_1=xy+ C z^2$, $q_2= yz$ with $C \neq 0$. Applying the projective 
transformation induced by the matrix $\diag(1,C,1)$ 
we can assume that $J$ has basis $\{xy + z^2,yz\}$. By construction $J$ satisfies condition (i)-(iii) of lemma 
\ref{nfdcrit}. We apply Corollary \ref{satcrit2} and set $f:=(xy+z^2)^2 +y^2z^2$. Since $J$ has codimension 1 in $J_{\T_2}$ 
it is enough to show that there exists no $\epsilon >0$ such that $g:=f-\epsilon z^4$ 
is nonnegative. In fact we have 
\[ g(1,-t^2,t)= t^6-\epsilon t^4, \]
which is negative for all $\epsilon>0$ and sufficiently small $t \in \R$. For $f=(xy+z^2)^2 +y^2z^2$ we denote $F_f$ and 
$J_F$ by $F_{\T_2}^{*}$ and $J_{\T_2}^{*}$ respectively. \\ 

Case 2: $\dim J =1$. \\
In this case we can assume that $J$ is spanned by a single form $p$. But for each $p \in \form_{2}$ we have 
$|\zero(p)|\leq 1$ or $|\zero(p)|=\infty$. Thus $J$ does not meet the requirements of lemma 
\ref{nfdcrit}.   \\ 
 
Since $\dim J_{\T_3}=\dim J_{\T_3}=2$ the preceding argument shows also that there exists no face $F \in \typeC$ with $F 
\subset F_{\T_3}$ or $F \subset F_{\T_4}$. In particular we have now obtained a complete description of $\typeC$.   
The following theorem summarizes the results of this section. 

\begin{theorem}
The set $\typeC$ splits in 4 equivalence classes with respective dimensions as follows 
\begin{center}
\begin{tabular}{l|c|c|c|c}
\emph{Face} $F$ &$F_\D$& $F_{\T_1}^*$ & $F_{\T_1}^{**}$ & $F_{\T_3}^*$  \\ \hline 
$\dim F$ &5& 6 & 3 & 3   \\ 
$\dim J_F$&3& 3 & 2& 2  \\ 
\end{tabular}
\end{center} 
\end{theorem}
\begin{proof}
The dimensions in the last line were obtained during the construction of the according faces. To obtain the dimension 
of $F$ we use the fact $J_F^{[2]}=\spann(F)$. 
It is left to show that these faces are pairwise inequivalent. However this is apparent because the $\G$--operation 
preserves the dimension and the number of common real zeros of a face. 
\end{proof} 

\begin{bem}
One can compute the entire list of faces of $\psd$ without assuming that a nonnegative form is a sum of squares. Since all 
extremal forms are perfect squares and every element in a cone is a finite sum of extremal elements, one obtains an 
alternative proof that every nonnegative ternary quartic is a sum of squares. 
\end{bem}

Now we have an entire list of faces of $\psd$. In the final two sections we will discuss exposednes of faces and inclusions 
between them. 

\section{Exposed faces}
We have already seen that all faces in $\typeA$ and faces which are equivalent to one of $F_{\emptyset}$, $F_{\S_1}$, $ 
F_{\S_2}$, $F_{\S_3}$ ,$F_{\S_4}$ or  $F_\Qu$ are exposed. We will now show that all remaining faces are not exposed. 

Each linear map $L\colon \form_{3,4} \longto \R$ which is nonnegative on $\psd$ induces a positive semidefinite 
symmetric bilinear form $B_L\colon \form_{3,2} \times \form_{3,2} \longto \R$ given by $B_L(f,g)=L(fg)$. 
We define $\ker(B_L):=\{f \in \form_{3,2} \colon \ B_L(f,g)=0, \  \forall g \in \form_{3,2}\}$. 

\begin{lemma}
Let $F$ be an exposed face defined by $L$. Then we have $\ker(B_L)=J_F$. 
\end{lemma}
\begin{proof}
Let $q \in \form_{3,2}$. Since $B_L$ is positive semidefinite we have 
\begin{align*}
q \in J_F &\iff q^2 \in F \\
          &\iff L(q^2)=0 \\
          &\iff B_L(q,q)=0 \\
          &\iff q \in \ker(B_L).
\end{align*}
\end{proof}  

\begin{lemma}
\label{exp}
Let $F \subseteq \psd$ a face. If there exists $f \in J_F$ and $g,c,d \in \form_{3,2}$ such that 
$0 \not \equiv c \equiv d \pmod{J_F}$ and $fg=cd$, then $F$ is not exposed. This is in particular the case if 
we find $f \in J_F$ and $c \notin J_F$ such that  $fg=c^2$.
\end{lemma}
\begin{proof}
Suppose $F$ is exposed and defined by the linear form $L$. 
We set $r:=c-d \in J_F$. Then we have 
\[ 0=B_L(f,g)=B_L(c,d)=B_L(d,d+r)=B_L(d,d) \]
which implies $d \in \ker B_L(d,d) =J_F$. This is a contradiction. 
\end{proof}

\begin{theorem}
All faces in $\typeC$ as well as faces which are equivalent to one of $F_{\T_1},F_{\T_2},F_{\T_3},F_{\T_4}$
are not exposed. 
\end{theorem}
\begin{proof}
We use lemma \ref{exp}. For $i=1,2,3$ we have $x^2 \in J_{\T_i}$ but $xz \notin J_{\T_i}$. Thus the faces 
$F_{\T_1},F_{\T_2}$ and $F_{\T_3}$ cannot be exposed. In order to show that $F_{\T_4}$ is not exposed we note 
that 
\[ yzx^2= \big(z(x-y)+yz+x(y-z)\big)\cdot \big(z(x-y)+yz\big). \]
Since $yz,x(y-z) \in J_{\T_4}$ but $z(x-y) \notin J_{\T_4}$ lemma \ref{exp} implies that $F_{\T_4}$ is not exposed. 
In the same way we conclude from $y^2 \in J_{\T_1}^*,J_{\T_1}^{**}$ and $yz \notin J_{\T_1}^*,J_{\T_1}^{**}$ that 
$F_{\T_1}^*$ and $F_{\T_1}^{**}$ are not exposed. The same argument applied to $y^2 \in J_\D$ and $xy \notin J_\D$ shows 
that $F_\D$ cannot be exposed. Last we note $0 \not\equiv -xy \equiv  z^2 \pmod{J_{\T_2}^*}$ but $yz \in J_{\T_2}$. 
Thus the equality $-(xz)(yz)=-xyz^2$ implies that $F_{\T_2}^*$ is not exposed.  
\end{proof}

\section{Inclusions}
\begin{defi}
We define a partial ordering $\Subset$ on $\mathcal{F}$ by 
\[ [F] \Subset [G]  \quad \text{if and only if there exists $\sigma \in \GL$ such that $F \subseteq \sigma G$}. \]
\end{defi}

Our goal is to give a complete description of all inclusions between equivalence classes of faces. In our preceding 
discussion we have chosen all representatives of faces in a way that makes most inclusions (and non-inclusions) obvious. 
We will discuss the ambiguous cases. 

\begin{lemma}
Let $F,G \subset \psd$ be faces with $\dim F=\dim G$. Then we have $F \Subset G$ if and only if $F \sim G$. 
\end{lemma}
\begin{proof}
Suppose $F \subset G$. Since $\dim F=\dim G$ we have $\spann(F)= \spann(G)$ and thus 
\[F=\spann(F)\cap \psd=\spann(G)\cap \psd = G.\] 
\end{proof}

To simplify the situation we first discuss the lattice of $\typeA$ relative to $\Subset$. 

\begin{theorem}
Relative to $\Subset$ the equivalence classes within $\typeA$ form following lattice. 
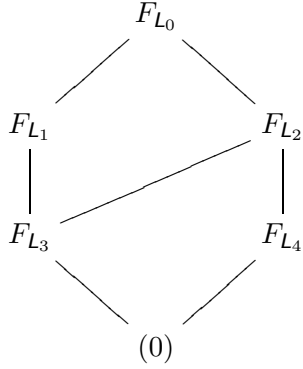
\label{binclusion}
\[ \begin{xy}
\xymatrix{ &F_{\Li_0} \ar@{-}[ld] \ar@{-}[rd] & \\
F_{\Li_1} \ar@{-}[d] & & F_{\Li_2} \ar@{-}[d] \ar@{-}[lld]  \\
F_{\Li_3} \ar@{-}[rd] & & F_{\Li_4} \ar@{-}[ld] \\
& (0) &  }
\end{xy}
\]
\end{theorem}

The following theorem clarifies the inclusions between faces in $\typeA$ and $\typeC$. 

\begin{lemma}
We have following inclusions 
\begin{enumerate}
\item $F_{\Li_2} \Subset F_{\T_1}^*$ but $F_{\Li_1} \not\Subset F_{\T_1}^*$.
\item $F_{\Li_4} \Subset F_{\T_1}^{**}$ but $F_{\Li_3} \not\Subset F_{\T_1}^{**}$.
\item $F_{\Li_3} \Subset F_{\T_2}^{*}$ but $F_{\Li_4} \not\Subset F_{\T_2}^{*}$. 
\item $F_{\Qu} \Subset 
F_{\T_1}^*,F_{\T_1}^{**},F_{\T_2}^*$.  
\end{enumerate}
\end{lemma}
\begin{proof}
First we note that for a face $F \in \psd$ and $(l,U) \in \typeA$ we have $(l,U) \subset F$ if and only if 
$F$ contains a form $l^2q$ with $q \in \psd$ and $\zero(q)=U$. 
\begin{enumerate}
\item Since $z^2,yz \in J_{\T_1}^*$ we have $f=z^2(y^2+z^2) \in F_{\T_1}^*$ and thus $F_{\Li_2} \Subset F_{\T_1}^*$.
Suppose now there exists a line $l$ and  $q$ a positive semidefinite quadratic form with $\zero(q)=s \notin \zero(l)$ such 
that $g=l^2q \in F_{\T_1}^*$. If $s \neq e_1$ then $g \in F_{\T_1}$ implies that $l=y$. But then we have $F_{\Li_0} \Subset 
F_{f+g}$ which is a contradiction to $F_{f+g} \subset F_{\T_1}^*$ and $\dim F_{\T_1}^* < \dim F_{\Li_0}$. Hence we must have 
$s=e_1$. Hence we must have $q=Ay^2+Bz^2$ with $A,B >0$. Since $e_1 \notin \zero(l)$ we obtain $\inp_g(e_1)=\emptyset$, 
which is a contradiction. 
\item 
Since $y^2 \in J_{\T_1}^{**}$ we have $y^4 \in F_{\T_1}^{**}$ and in particular $F_{\Li_4} \Subset F_{\T_1}^{**}$. 
Let $l_1,l_2$ be two different lines in $\proj^2(\R)$ with intersection $s$ and set $f=l_1^2 l_2^2$. 
Suppose $f \in F_{\T_1}^{**}$. If $s \neq e_1$, then $l_1$ or $l_2$ must contain $e_1 \vee e_3$. Thus we can assume 
$l_1=y$. But then $F_{\Li_2} \Subset F_{f+y^4} \subset F_{\T_1}^{**}$. This is a contradiction since $\dim F_{\Li_2}=\dim 
F_{\T_1}^{**}$ but $F_{\T_1}^{**} \neq F_{\Li_2}$. Thus we must have $s=e_1$ in particular $\ord_{e_1}(f)=4$. Recall that 
$J_{\T_1}^{**}$ has basis $\{y^2,xy+z^2\}$.    Since $y^2$ is the only element having order 2 in $e_1$, we conclude that 
$y^4$ is the only element  in $F_{\T_1}^{**}$ with $\ord_{e_1}=4$.  Therefore we 
have $F_{\Li_3} \not\Subset F_{\T_1}^{**}.$ 
\item 
Since $yz \in J_{\T_2}^{*}$ we have $F_{\Li_3} \Subset F_{\T_2}^{*}$. Suppose there exists a line $l$ such that 
$f:=l^4 \in F_{\T_2}^{*}$. Since $l$ has to vanish at $e_1$ and $e_2$ we must have $l=z$. But then $z^4 \in 
F_{\T_3}^{*}$, which is not the case. 
\item 
The form $(xy+z^2)^2$ is a square of a regular indefinite quadratic form contained in $F_{\T_1}^*$,$F_{\T_1}^{**}$ and 
$F^*_{\T_2}$.
\end{enumerate}
\end{proof}

Eventually we draw a picture of the lattice of equivalence classes of faces with respective dimensions.  
To keep the picture clear we omit inclusions within $\typeA$, which are given by theorem 
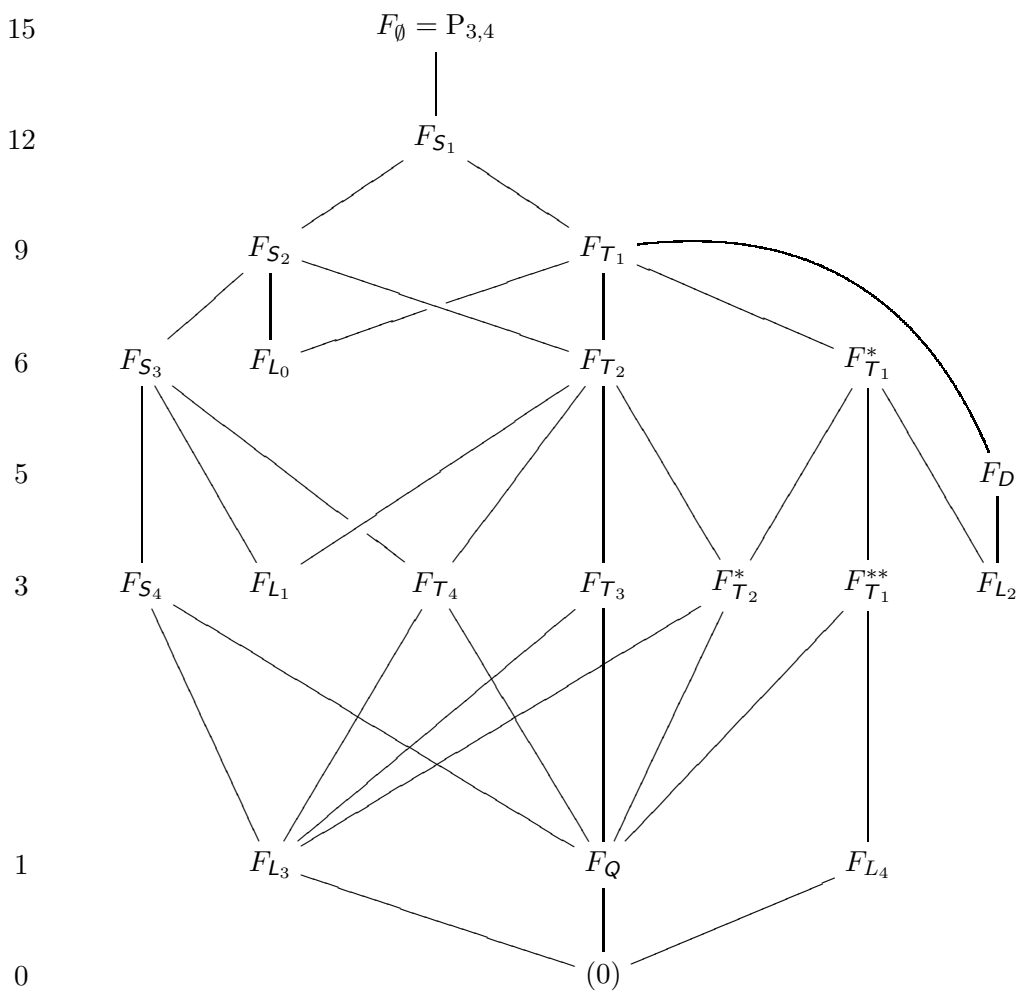
\ref{binclusion}.   

\[
\begin{xy}
\xymatrix{ 
           15&&& F_{\emptyset}=\psd \ar@{-}[d] &&&& \\
           12&&& F_{\S_1} \ar@{-}[ld] \ar@{-}[rd] \\
           9&& F_{\S_2} \ar@{-}[ld] \ar@{-}[d] \ar@{-} [rrd]  && F_{\T_1} \ar@{-}[lld]|!{[d];[ll]}\hole 
\ar@{-}[drr]\ar@{-}[d]\ar@{-}@/^3pc/[rrrdd]&&&\\
           6&F_{\S_3} \ar@{-}[dd] \ar@{-}[ddr] \ar@{-}[ddrr]|!{[rdd];[rrr]}\hole& F_{\Li_0} && F_{\T_2}\ar@{-}[ddll] \ar@{-}[ddl] \ar@{-}[dd] \ar@{-}[ddr]  && F_{\T_1}^* \ar@{-}[dd] \ar@{-}[ddl] \ar@{-}[ddr] & \\
           5&&&  & & & &F_{\D} \ar@{-}[d]    \\
           3&F_{\S_4} \ar@{-}[dddr]\ar@{-}[dddrrr] & F_{\Li_1} & F_{\T_4} \ar@{-}[dddl] \ar@{-}[dddr]& F_{\T_3} \ar@{-}[dddll] \ar@{-}[ddd] & F_{\T_2}^* \ar@{-}[lllddd] \ar@{-}[dddl]  & F_{\T_1}^{**}\ar@{-}[dddll] \ar@{-}[ddd]  & F_{\Li_2}  \\
           &&&&&&&\\
           &&&&&&&\\
         1  && F_{\Li_3} \ar@{-}[rrd] & &  F_{\Qu} \ar@{-}[d] & &  F_{L_4} \ar@{-}[dll] & \\
         0  &&&&(0)& &&}
           
\end{xy}
\]


\begin{thebibliography}{20}

\bibitem[Ba]{Bar} A. Barvinok, \textit{A Course in Convexity}, Graduate Studies in Mathematics, vol. 54, American 
Mathematical Society, Providence, RI 2002.
\bibitem[Bl]{Blekh} G. Blekherman, \textit{Dimensional differences between nonnegative polynomials and sums of squares}, 
arXiv:0907.1339.
\bibitem[BCR]{bcr} J. Bochnak, M. Coste, M.F. Roy, \textit{Real algebraic geometry}, Springer 1998. 
\bibitem[CLR]{CLR} M.-D. Choi, T.-Y. Lam, B. Reznick, \textit{Real Zeros of Positive Semidefinite Forms I},
Mathematisches Zentralblatt 1-26, Springer 1980. 
\bibitem[Fu]{fulton} W. Fulton,  \textit{Algebraic Curves}, Benjamin New York, 1969
\bibitem[Hi]{hilbert} D. Hilbert, \textit{\"Uber die Darstellung definiter Formen als Summe von Formenquadraten},
Math. Ann. 32, 342–350 (1888).

\bibitem[PD]{Prest} A. Prestel and C. Delzell, \textit{Positive Polynomials. From Hilbert's 17th Problem to Real Algebra.} 
Springer Monographs in Mathematics, Springer-Verlag, Berlin (2001).
\bibitem[Re]{Rez2} B. Reznick, \textit{On Hilbert's construction of positive polynomials}, Submitted for publication, 
arXiv:0707.2156.


\end{thebibliography}
\end{document}